\documentclass[11pt,reqno,fleqn]{article}
\usepackage{amsmath, amsthm}\usepackage{enumerate}\usepackage{amssymb}\usepackage{bbold}
\usepackage{bbm}\usepackage{dsfont}\usepackage{color}\usepackage{xcolor}
\usepackage[top=20mm,bottom=20mm,left=20mm,right=20mm]{geometry}
\usepackage[T1]{fontenc}
\usepackage{times}
\usepackage{multicol}
\usepackage{graphicx}
\usepackage{graphics}
\usepackage{enumitem}
\usepackage{cite}
\usepackage{mathtools}

\makeatletter
\let\@fnsymbol\@arabic
\makeatother
\newcommand\blfootnote[1]{%
  \begingroup
  \renewcommand\thefootnote{}\footnote{#1}%
  \addtocounter{footnote}{-1}%
  \endgroup
}
\theoremstyle{definition}
\newtheorem{theorem}{Theorem}[section]
\numberwithin{equation}{section}
\newtheorem{definition}[theorem]{Definition}
\newtheorem{corollary}[theorem]{Corollary} 
\newtheorem{proposition}[theorem]{Proposition}
\newtheorem{remark}[theorem]{Remark}
\theoremstyle{definition}
\newtheorem{example}[theorem]{Example}
\usepackage{etoolbox}
\makeatletter
\patchcmd{\@maketitle}{\begin{center}}{\begin{flushleft}}{}{}
\patchcmd{\@maketitle}{\begin{tabular}[t]{c}}{\begin{tabular}[t]{@{}l}}{}{}
\patchcmd{\@maketitle}{\end{center}}{\end{flushleft}}{}{}
\makeatother
\allowdisplaybreaks

\DeclareMathOperator{\sto}{\xrightarrow[]{st-o}}
\DeclareMathOperator{\stuo}{\xrightarrow[]{st-uo}}
\DeclareMathOperator{\pc}{\xrightarrow[]{p}}
\DeclareMathOperator{\std}{\downarrow^{st}}
\DeclareMathOperator{\oc}{\xrightarrow[]{o}}
\DeclareMathOperator{\uoc}{\xrightarrow[]{uo}}
\DeclareMathOperator{\stpd}{\downarrow^{st_p}}
\DeclareMathOperator{\stpc}{\xrightarrow[]{st_p}}
\DeclareMathOperator{\stupc}{\xrightarrow[]{st-up}}

\begin{document}
\title{Statistically unbounded $p$-convergence in lattice-normed Riesz spaces}
\author{Abdullah Ayd\i n\\ \small Department of Mathematics, Mu\c{s} Alparslan University, Mu\c{s}, 49250, Turkey\\ \small
a.aydin@alparslan.edu.tr}
\date{{\small }}
\maketitle
\noindent

\noindent\textbf{Abstract:}
The statistically unbounded $p$-convergence is an abstraction of the statistical order, unbounded order, and $p$-convergences. We investigate the concept of the statistically unbounded convergence on lattice-normed Riesz spaces with respect to statistical p-decreasing sequences. Also, we get some relations between this concept and the other kinds of statistical convergences on Riesz spaces.
\noindent

\section{Introduction}
\label{Sec:1}
A Riesz space is an ordered vector space, which was introduced by Riesz in \cite{Riez}. Most of the spaces encountered in the analysis are Riesz spaces. Moreover, Riesz spaces have many applications in measure theory, operator theory, optimization, problems of Banach spaces, measure theory, and applications in economics (cf. \cite{AB,ABPO,AlTo,LZ,Za}).

The {\em order convergence} is one of the fundamental concepts in the study of Riesz spaces, which is not topological in general (cf. \cite{Gor}). However, even without a topological structure, several kinds of continuous	operators can be defined. Another fundamental concept in the theory of Riesz spaces is {\em unbounded order convergence}, which was firstly introduced in \cite{N} under the name individual convergence. A lot of work has been done since then (cf. \cite{AEEM1,DOT,GTX,T}). The unbounded order convergence will be the basic tool of this paper. Lattice-valued norms on vector lattices provide natural and efficient tools in the theory of vector lattices. It is enough to mention the theory of lattice-normed vector lattices (cf. \cite{E,K,KK}).

The theory of the {\em statistical convergence} is an active area of research, which is a generalization of the ordinary convergence of a real sequence, and the idea of the statistical convergence was firstly introduced by Zygmund \cite{Zygmund}. After then, Fast \cite{Fast} and Steinhaus \cite{St} independently improved that idea. Several applications and generalizations of the statistical convergence of sequences have been investigated by several authors (cf. \cite{Aydn1,Aydn2,Aydn3,Ercan,Fridy,SP}. Most of the works on the theory of the statistical convergence are handled with respect to a given topology. But, Ercan introduced a concept of the statistical convergence in Riesz spaces without topology in \cite{Ercan}.Then Şençimen and Pehlivan extended this concept to Riesz spaces with respect to the order convergence; see \cite{SP}. Recently, Ayd\i n et al., have investigated some studies about the statistical convergence on Riesz spaces, Riesz algebras, and locally solid Riesz spaces; see \cite{Aydn1,Aydn2,Aydn3}. In the present paper, we aim to introduce the concept of the statistically unbounded $p$-convergence on lattice-normed spaces and illustrate the usefulness of lattice-valued norms for the investigation of different types of statistical convergence in Riesz spaces.

The structure of the paper is as follows. In Section 2, we give several notions related to lattice-normed spaces and statistical convergence. In Section 3, we introduce the concept of the statistically unbounded convergence in lattice-normed spaces. In the last section, we show some main results.
\section{Preliminaries}\label{sec2}

We begin the section with some basic concepts related to the theory of Riesz space and refer to \cite{AB,ABPO,K,KK,LZ,Za} for more details.

\begin{definition}
	A real-valued vector space $E$ with a partial order relation "$\leq$" on $E$ (i.e. it is an antisymmetric, reflexive and transitive relation) is called an {\em ordered vector space} whenever, for every $x,y\in E$, we have
	\begin{enumerate}
		\item[(a)] $x\leq y$ implies $x+z\leq y+z$ for all $z\in E$,
		\item[(b)] $x\leq y$ implies $\lambda x\leq \lambda y$ for every $0\leq \lambda \in \mathbb{R}$.
	\end{enumerate}	
\end{definition}
An ordered vector space $E$ is called a {\em Riesz space} or a {\em vector lattice} if, for any two vectors $x,y\in E$, the infimum and the supremum
$$
x\wedge y=\inf\{x,y\} \ \ \text{and} \ \ x\vee y=\sup\{x,y\}
$$
exist in $E$, respectively. A Riesz space is called {\em Dedekind complete} if every nonempty bounded above subset has a supremum (or, equivalently, whenever every nonempty bounded below subset has an infimum). For an element $x$ in a Riesz spaces $E$, {\em the positive part}, {\em the negative part}, and {\em module} of $x$, respectively 
$$
x^+:=x\vee0, \ \ \ x^-:=(-x)\vee0\ \ and \ \ |x|:=x\vee (-x).
$$ 
In the present paper, the vertical bar $|\cdot|$ of elements of the Riesz spaces will stand for the module of the given elements. A Riesz space $E$ is called {\em Archimedean} whenever $\frac{1}{n}x\downarrow0$ holds in $E$ for each $x\in E_+$. Unless otherwise stated, we assume that all vector lattices are real and Archimedean. 

A sequence $(x_n)$ in a Riesz space $E$ is said to be {\em increasing} whenever $x_1 \leq x_2\leq\cdots$, and {\em decreasing} if $x_1 \geq x_2\geq \cdots$. Then we denote them by $x_n\uparrow$ and $x_n\downarrow$ respectively. Moreover, if $x_n\uparrow$ and $\sup x_n=x$ then we write $x_n\uparrow x$. Similarly, if $x_n\downarrow$ and $\inf x_n=x$ then we write $x_n\downarrow x$. Then we call that $(x_n)$ is increasing or decreasing as {\em monotonic}. Moreover, we remind that a sequence $(x_n)$ in a Riesz space $X$ is called
\begin{enumerate}
	\item[-] {\em order convergent} to $x\in X$ $($$x_n\oc x$, for short$)$ whenever there exists a sequence $(q_n)\downarrow0$ in $X$ such that $|x_n-x|\le q_n$ for all $n\in \mathbb{N}$,
	\item[-] {\em unbounded order convergent} (or $uo$-convergent, for short) to $x\in X$ if $|x_n-x|\wedge u\oc0$ for every $u\in X_+$, in this case, we write $x_n\uoc x$.
\end{enumerate}
It is clear that order convergence implies $uo$-convergence because $|a-b|\wedge u\leq|a-b|$ for any $u\geq 0$. But, the converse need not be true. For example, consider the sequence $(e_n)$ of the standard unit vectors in $c_0$. Then $e_n\uoc 0$, but does not converge in order because $(e_n)$ is not order bounded in $c_0$. In $L_p$-spaces, where $1\leq p<\infty$, the $uo$-convergence of sequences is equivalent to almost everywhere convergence (cf. \cite{GTX}). Also, if $E:=\ell_p$, $1\leq p<\infty$, $c_0$ or $c$ then $uo$-convergence is equivalent to the coordinate-wise convergence (cf. \cite{DOT}).

\begin{definition}
	Let $X$ be a vector space and $E$ be a Riesz space. Then $p:X \to E_+$ is called an {\em $E$-valued vector norm} whenever it satisfies the following conditions:
	\begin{enumerate}
		\item $p(x)=0\Leftrightarrow x=0$;
		\item $p(\lambda x)=|\lambda|p(x)$ for all $\lambda\in\mathbb{R}$;
		\item $p(x+y)\leq p(x)+p(y)$ for all $x,y\in X$.
	\end{enumerate}
\end{definition}
Then the triple $(X,p,E)$ is called a {\em lattice-normed space}, abbreviated as $LNS$. 
If, in addition, $X$ is a Riesz space and the vector norm $p$ is monotone (i.e., $|x|\leq |y|\Rightarrow p(x)\leq p(y)$ holds for all $x,y\in X$) then the triple $(X,p,E)$ is called a {\em lattice-normed vector lattice} or a {\em lattice-normed Riesz space}. We abbreviate it as $LNRS$.
While dealing with $LNRS$s, we shall keep in mind also the following examples.
\begin{example}\label{ExLNVL_2}
	Let $X$ be a normed space with a norm $\|\cdot\|$. Then $(X,\|\cdot\|,{\mathbb R})$ is an $LNS$ . 
\end{example}

\begin{example}\label{ExLNVL_3}
	Let $X$ be a Riesz space. Then $(X,|\cdot|,X)$ is an $LNRS$ . 
\end{example}

We abbreviate the convergence $p(x_n-x)\oc0$ as $x_n\pc x$ and say in this case that $(x_n)$ {\em $p$-converges} to $x$ in an $LNS$ $(X,p,E)$. Moreover, an $LNRS$ $(X,p,E)$ is called {\em $op$-continuous} if $x_n\oc0$ implies $p(x_n)\oc0$. In an $LNS$ $(X,p,E)$ a subset $A$ of $X$ is called {\em $p$-bounded} if there exists $e\in E$ such that $p(a)\leq e$ for all $a\in A$. A vector $e\in X$ is called a {\em $p$-unit} if, for any $x\in X_+$, $p(x-ne\wedge x)\xrightarrow{o}0$. We refer the reader for more information on $LNS$s to \cite{AEEM1,E,KK,K}.

Now, we recall some basic properties of the concepts related to the statistical convergence. Consider a set $K$ of positive integers. Then the {\em natural density} of $K$ is defined by
$$
\delta(K):=\lim_{n\rightarrow \infty}\frac{1}{n}\left \vert \left \{  k\leq n:k\in
K\right \}  \right \vert,
$$
where the vertical bar of sets will stand for the cardinality of the given sets. We refer the reader to an exposition on the natural density of sets to \cite{Fast,Fridy}. In the same way, a sequence $x=(x_{k})$ is called \textit{statistical convergent} to $L$ provided that
$$
\lim_{n\rightarrow \infty}\frac{1}{n}\left \vert \left \{  k\leq n:\left \vert
x_{k}-L\right \vert \geq \varepsilon \right \}  \right \vert =0
$$
for each $\varepsilon>0$. Then it is written by $S-\lim x_{k}=L$.
We take the following notions from \cite{Aydn3,SP}. Let $(x_n)$ be a sequence in a Riesz space $E$. Then $(x_n)$ is called
\begin{enumerate}
	\item[-] {\em statistically order decreasing} to $0$ if there exists a set $K=\{n_1<n_2<\cdots\}\subseteq\mathbb{N}$ with $\delta(K)=1$ such that $(x_{k_n})$ is decreasing and $\inf\limits_{k_n\in K}(x_{k_n})=0$, and so, it is abbreviated as $x_n\std 0$,
	
	\item[-] {\em statistically order convergent} to $x\in E$ if there exists a sequence $q_n\std0$ with a set $K=\{n_1<n_2<\cdots\}\subseteq\mathbb{N}$ such that $\delta(K)=1$ and $|x_{k_n}-x|\leq q_{k_n}$ for every $k_n\in K$, and so, we write $x_n\sto x$,
	
	\item[-] {\em statistically unbounded order convergent} to $x\in E$ if, for every $u\in E_+$, there exists a sequence $q_n\downarrow^{st}0$ and a subset $K$ of the natural numbers such that $\delta(K)=1$ and 
	$$
	|x_{k_n}-x|\wedge u\le q_{k_n}
	$$
	for all $k_n\in K$, and so, we abbreviate it as $x_n\stuo x$.
\end{enumerate}
Moreover, we take the following notions from \cite{AYE}. A sequence in an $LNS$ $(X,p,E)$ is said to be
\begin{enumerate}
	\item[-] {\em statistically $p$-decreasing} to $0$ if there exists a set $K=\{n_1<n_2<\cdots\}\subseteq\mathbb{N}$ such that $\delta(K)=1$ and $p(x_{k_n})\downarrow 0$ on $K$, and so, we abbreviate it as $x_n\stpd0$,
	\item[-] {\em statistically $p$-convergence} to $x$ if there exists a sequence $q_n\stpd0$ in $X$ with a set $K=\{n_1<n_2<\cdots\}\subseteq\mathbb{N}$ such that $\delta(K)=1$ and $p(x_{k_n}-x)\leq q_{k_n}$ for all $k_n\in K$, and so, we write $x_n\stpc x$. 
\end{enumerate}
\section{Statistically unbounded $p$-convergence}\label{sec3}
We introduce the statistically unbounded convergence on $LNRS$s in this section. Recall that a sequence $(x_n)$ in an $LNRS$ $(X,p,E)$ is said to be {\em unbounded $p$-convergent} to $x\in X$ (shortly,  
$x_n\xrightarrow{up}x$), if $p(|x_n-x|\wedge u)\oc0$ for all $u\in X_+;$ see \cite[Def.6]{AEEM1}. The following definition is motivated by the notion of $up$-convergence.
\begin{definition}
	Let $(X,p,E)$ be an $LNRS$ and $(x_n)$ be a sequence in $X$. Then $(x_n)$ is said to be {\em statistically unbounded $p$-convergence} to $x$ if, for any positive element $u\in X_+$, there exists a sequence $q_n\stpd0$ in $X$ with an index set $K=\{n_1<n_2<\cdots\}\subseteq\mathbb{N}$ such that $\delta(K)=1$ and 
	$$
	p(|x_{k_n}-x|\wedge u)\leq q_{k_n}
	$$
	for all $k_n\in K$. Then we write $x_n\stupc x$. 
\end{definition}

It is clear that a sequence $x_n\stupc x$ holds in an $LNRS$ $(X,p,E)$ if and only if $(x_n-x)\stupc 0$ if and only if $p(|x_n-x|\wedge u)\sto 0$ in $E$ for all $u\in X_+$.
\begin{example}
	Let $X^*$ be the algebraic dual of a Riesz space $X$ and $Y$ be a sublattice of $X^*$ such that $\left\langle X,Y\right\rangle$ is a dual system. Then define an $LNRS$ $(X,p,{\mathbb R}^Y)$ with a lattice norm $p:X\to {\mathbb R}^Y$ denoted by $p(x)[f]:=|f|(|x|)$. Take a sequence $(x_n)$ in $X$. Thus, it is clear that $|x_n|\wedge u\xrightarrow{|\sigma|(X,Y)}0$ for every $u\in X_+$, and so, we have $x_n\stupc 0$.
\end{example}

The next observations follow directly from the basic definitions and results, and so, we omit their proofs.
\begin{remark}\label{main remark}\ 
	\begin{enumerate}
		\item[(i)] It follows from the inequality $|a-b|\wedge u\leq |a-b|$ in Riesz spaces and by the monotonicity of $p$ in $LNRS$s that statistical $p$-convergence implies $st$-$up$-convergence.
		
		\item[(ii)] Statistically unbounded $p$-convergence coincides with the notion of statistical $p$-convergence for order bounded sequences.
		
		\item[(iii)] Consider an $LNRS$ $(L_p(\mu),\lVert\cdot\rVert,\mathbb{R})$ for $1\leq p<\infty$, where $\mu$ is a finite measure. Then the convergence in measure implies $st$-$up$-convergence of sequences in $L_p(\mu)$; see \cite{T}.
		
		\item[(iv)] Take the $LNRS$s $(L_p,\lvert\cdot\rvert,L_p)$ for $1\leq p<\infty$. Then the abstraction of a.e.-convergence in $L_p$-spaces implies $st$-$up$-convergence of sequences in $L_p$; see \cite{GTX}.
		
		\item[(v)] Consider an $LNRS$ $(E,|\cdot|,E)$ for an arbitrary Riesz space $E$. Then statistically unbounded order convergence and $st$-$up$-convergence agree.
		
		\item[(vi)] Every $up$-convergent sequence is statistically unbounded $p$-convergent to its $up$-limit.
		
		\item[(vii)] A statistical $p$-decreasing sequence is $st$-$up$-convergent.
		
		\item[(viii)] Take an $LNRS$ $(X,p,E)$, where $E$ is $\ell_p$ ($1\leq p<\infty$), $c_0$ or $c$. Then the coordinate-wise convergence of a sequence implies $st$-$up$-convergence.
		
		\item[(ix)] An order convergent sequence in $op$-continuous $LNRS$s is $st$-$up$-convergent to its order limit.
	\end{enumerate}
\end{remark}

The converse of the properties of Remark \ref{main remark} need not be true in general. To see some of them, we give the following examples.
\begin{example}
	Consider the $LNRS$ $(c_0,|\cdot|,c_0)$, where $c_0$ is the set of all convergent to zero real sequences. Take the sequence $(e_n)$ of the standard unit vectors in $c_0$. Then $(e_n)$ is statistically unbounded $p$-convergent to $0$. However, it is not $p$-convergent because it is not $p$-bounded in $c_0$.
\end{example}

\begin{example}\label{exam 2}
	Let's consider the $LNRS$ $(\mathbb{R}^2,|\cdot|,\mathbb{R}^2)$, where Euclidean space $\mathbb{R}^2$ with the coordinate-wise ordering. Take a sequence $(x_n)$ in $\mathbb{R}^2$ denoted by
	$$
	x_n:=
	\begin{cases} 
		(0,n), & k=n^3 \\
		(1,\frac{1}{n}), & k\neq n^3
	\end{cases}
	$$
	for all $n$, where $k\in\mathbb{N}$. Thus, $(x_n)$ is statistically unbounded $p$-convergent to $(1,0)\in \mathbb{R}^2$. But, it does not $up$-convergent.
\end{example}

\begin{example}\label{monotone not conversely}
	Let's take the $LNRS$ $c$ of all convergent real sequences. Consider a sequence $(x_n)$ in $LNRS$ $(c,|\cdot|,c)$ denoted by $x_n:=(x_k^n)=(x_1^n,x_2^n,\cdots,x_k^n,\cdots)\in E$ such that
	$$
	x_k^n:=
	\begin{cases} 
		1, & k=n^2 \\
		\frac{1}{k+1}, & k\neq n^2
	\end{cases}
	$$
	for all $n,k\in\mathbb{N}$. Then it is clear that $x_n\stupc 0$. Observe that the whole sequence $(x_n)$ is not monotonic.
\end{example}

It is clear from Example \ref{exam 2} that a statistical $p$-monotone decreasing sequence need not be monotone convergent in general. It is well known that a subsequence of order convergent sequence is order convergent to its order limit. However, this need not be true for $st$-$up$-convergence. To illustrate this, we consider the following example.
\begin{example}\label{subseq example}
	Take the Riesz space $\mathbb{R}$ and a sequence $(x_n)$ in $\mathbb{R}$ defined by
	$$
	x_n:=
	\begin{cases} 
		(-1)^n(2n+1), & k=n^5 \\
		\frac{1}{2n+1}, & k\neq n^5
	\end{cases}
	$$
	for all $n$, where $k\in\mathbb{N}$. Fix $u\in \mathbb{R}_+$. Thus, for the sequence $q_n:=\frac{1}{n+1}\downarrow 0$, we have $|x_n|\wedge u\leq q_n$ for all $n\in\mathbb{N}$ such that $n\neq k^5$ for all $k\in\mathbb{N}$. Hence, we obtain $x_n\stupc 0$. But, by choosing a subsequence $(x_{m_n})$ of $(x_n)$ such that $m_n=k^5$ for some $k\in \mathbb{N}$, it is clear that $(x_{m_n})$ is not $st$-$up$-convergent.
\end{example}

By the following sense, we see that the lattice operations in an $LNRS$ are statistically unbounded $p$-continuous.
\begin{theorem}\label{LO are continuous}
	If $x_n\stupc x$ and $y_n\stupc y$ satisfy then $x_n\vee y_n\stupc x\vee y$ in $LNRS$s.
\end{theorem}

\begin{proof}
	Let $x_n\stupc x$ and $y_n\stupc y$ be in an $LNRS$ $(X,p,E)$. Take an arbitrary $u\in E_+$. Then there exist sequences $q_n\stpd0$ and $t_n\stpd0$ with $\delta(K)=\delta(M)=1$ such that $p|(x_{k_n}-x|\wedge u)\leq q_{k_n}$ and $p(|x_{m_n}-y|\wedge u)\leq t_{m_n}$ for each $k_n\in K$ and $m_n\in M$. By using the inequality $|a\vee b-a\vee c|\leq|b-c|$ (cf. \cite[Thm.1.9(2)]{ABPO}) in Riesz spaces, we have
	\begin{eqnarray*}
		p(|x_n\vee y_n-x\vee y|\wedge u)&\leq& p(|x_n\vee y_n-x_n\vee y|\wedge u)+p(|x_n \vee y-x\vee y|\wedge u)\\&\leq& p(|y_n-y|\wedge u)+p(|x_n-x|\wedge u)
	\end{eqnarray*}
	for every $n\in\mathbb{N}$. So, it follows that
	$$
	p(|x_{j_n}\vee y_{j_n}-x\vee y|\wedge u)\leq p(|y_{j_n}-y|\wedge u)+p(|x_{j_n}-x|\wedge u)\leq q_{j_n}+t_{j_n}
	$$
	for each $j_n\in J$, where $J:=K\cap J$. Therefore, we have $x_n\vee y_n\stupc x\vee y$ by $(q_{j_n}+t_{j_n})\downarrow0$.
\end{proof}

\begin{theorem}\label{basic properties}
	Let $(X,p,E)$ be an $LNRS$, $x_n\stupc x$ and $y_n\stupc y$. Then the following statements hold;
	\begin{enumerate}
		\item[(i)] if $x_n\stupc x$ and $x_n\stupc z$ then $x=y$;
		\item[(ii)] $\lambda x_n+\beta y_n\stupc \lambda x+\beta y$ for all $\lambda,\beta\in\mathbb{R}$;
		\item[(iii)] $|x_n|\stupc |x|$;
		\item[(iv)] $x_n^+\stupc x^+$;
		\item[(v)] if $x_n\geq y_n$ for all $n\in\mathbb{N}$ then we have $x\geq y$.
	\end{enumerate}
\end{theorem}

\begin{proof}
	The properties $(iii)$ and $(iv)$ are results of Theorem \ref{LO are continuous}, and so, we omit their proof. Moreover, $(v)$ can be obtained by applying $(iv)$, and $(ii)$ has a straightforward proof.
	
	To prove $(i)$, consider that a sequence $(x_n)$ in an $LNRS$ $(X,p,E)$ satisfies $x_n\stupc x$ and $x_n\stupc y$. Fix a positive element $u\in X_+$. Then there exist sequences $q_n\stpd0$ and $t_n\stpd0$ with index sets $K,M\subseteq\mathbb{N}$ such that $\delta(K)=\delta(M)=1$ and 
	$$
	p(|x_{k_n}-x|\wedge u)\leq q_{k_n} \ \ \text{and} \ \ p(|x_{m_n}-y|\wedge u)\leq t_{m_n}
	$$
	hold for all $k_n\in K$ and $m_n\in M$. Take $J:=K\cap M$, and so, we have $\delta(J)=1$. It follows $p(|x_{j_n}-x|\wedge u)\leq q_{j_n}$ and $p(|x_{j_n}-y|\wedge u)\leq r_{j_n}$ for each $j_n\in J$. Thus, we observe
	$$
	0\leq p(|x-y|\wedge u)\leq p(|x_{m_n}-x|\wedge u) + p(|x_{m_n}-y|\wedge u) \leq q_{j_n}+r_{j_n}
	$$
	for every $j_n\in J$. Thus, $(q_{j_n}+r_{j_n})\downarrow 0$ on $J$, and so, we obtain $|x-y|\wedge u=0$. Since $u$ is arbitrary, we get $x=y$.
\end{proof}

By applying Theorem \ref{basic properties}$(v)$, we give the following result.
\begin{corollary}
	Let $(X,p,E)$ be an $LNRS$ and $(x_n)$, $(y_n)$ and $(z_n)$ be sequences in $X$ such that $x_n\leq y_n\leq z_n$ for all $n\in\mathbb{N}$. Then $x_n\stpc x$ and $z_n\stpc x$ implies $y_n\stpc x$.
\end{corollary}

In the following theorem, we give a relation between $st$-$up$- and order convergence.
\begin{theorem}\label{$up$-sup-inf}
	Any monotone $st$-$up$-convergent sequence in an $LNRS$ is order convergent to its $st$-$up$-limit.
\end{theorem}

\begin{proof}
	Suppose that $x_n\stupc x$ and $x_n\downarrow$ hold in an $LNRS$ $(X,p,E)$. Then, for a fixed arbitrary $m$, we have $x_m-x_n\in X_+$ for each $n\ge m$. Thus, by using Theorem \ref{basic properties}, we get $x_m-x_n\stupc x_m-x\in X_+$. So, it is clear that $x$ is a lower bound of $(x_n)$ because $m$ is arbitrary. Choose another lower bound $w$ of $(x_n)$, i.e., $x_n\geq w$ for all $n$. Then, again by applying Theorem \ref{basic properties}, we obtain $x_n-x\stupc x-w\in X_+$. Hence, we have $w\leq x$, i.e., we get $x_n\downarrow x$.
\end{proof}

It follows from Remark \ref{main remark}$(i)$ that statistical $p$-convergence implies $st$-$up$-convergence. For the converse, we give the following result.
\begin{theorem}\label{st up implies st po}
	Let $(x_n)$ be a monotone and $x_n\stupc x$ sequence in an $LNRS$ $(X,p,E)$. Then $x_n\stpc x$.
\end{theorem}

\begin{proof}
	Assume that $0\le x_n\uparrow$ holds in $X$ by without loss of generality. It follows from Theorem \ref{$up$-sup-inf} that $x_n\oc x$ for some $x\in X$. Thus, we get $0\leq x-x_n\leq x$ for all $n\in \mathbb{N}$. Fix $u\in X_+$. Then we have $p((x-x_n)\wedge u)\sto0$ because of $x_n\stupc x$. If, in particular, we choose $u$ as $x\in X_+$ then we get $p(x-x_n)=p((x_n-x)\wedge x)\sto 0$, i.e., we have $x_n\stpc x$.
\end{proof}	
\section{Main Results}\label{sec4}
Recall that if $(x_n)$ is a sequence satisfying the property $P$ for all $n\in \mathbb{N}$ except a set of natural density zero then we say that $(x_n)$ satisfies the property $P$ for almost all $n$, and it is abbreviated by a.a.n.; see \cite{Fridy}.
\begin{theorem}
	A sequence $x_n\stupc x$ holds in $LNRS$s if and only if there is another sequence $(y_n)$ such that $x_n=y_n$ for a.a.n and $y_n\stupc x$.
\end{theorem}

\begin{proof}
	Assume that there is a sequence $(y_n)$ in an $LNRS$ $(X,p,E)$ such that $x_n=y_n$ for a.a.n and $y_n\stupc x$. Take an arbitrary $u\in X_+$.  Thus, there exists a sequence $q_n\stpd0$ with a set $\delta(K)=1$ such that $p(|y_{k_n}-x|\wedge u)\leq q_{k_n}$ for all $k_n\in K$. It follows from $x_n=y_n$ for a.a.n that the index sets $\{k_n\in K:p(|x_{k_n}-x|\wedge u)\leq q_{k_n}\}$ and $\{k_n\in K:p(|y_{k_n}-x|\wedge u)\leq q_{k_n}\}$ are a.a.n. Therefore, there is a subset $J$ of $K$ such that $\delta(J)=1$ and $p(|x_{j_n}-x|\wedge u)\leq q_{j_n}$ for each $j_n\in J$. So, we have $x_n\stupc x$.
\end{proof}

In the general case, Example \ref{subseq example} shows that a subsequence of $st$-$up$-convergent sequence need not be $st$-$up$-convergent. But, in the following work, we give a positive result for this.
\begin{theorem}
	Let $(X,p,E)$ be an $LNRS$ and $(x_n)$ be a sequence in $X$. If $x_n\stupc x$ and $0\leq x_n\downarrow$ then every subsequence $(x_{k_n})$ with an index set $\delta(\{k_1,k_2,\cdots,k_i,\cdots\})=1$ is $st$-$up$-convergent to $x$.
\end{theorem}

\begin{proof}
	Assume that $x_n\downarrow$ and $x_n\stupc x$ in an $LNRS$ $(X,p,E)$. Then, for any $u\in X_+$, there exists a sequence $q_n\stpd0$ with a set $\delta(K)=1$ such that $p(|x_{k_n}-x|\wedge u)\leq q_{k_n}$ for all $k_n\in K$. Thus, it can be seen that $x_{k_n}\stupc x$ holds. On the other hand, it follows from Theorem \ref{st up implies st po} that $x_n\stpc x$, i.e., $p(x_n-x)\sto 0$. Thus, we have $p(x_n-x)\std 0$ because of the monotonicity of $p$ and $(x_n-x)\downarrow$. Then, for an arbitrary $M=\{m_1,\cdots,m_n,\cdots\}\subseteq\mathbb{N}$ such that $\delta(M)=1$ and $M\neq K$, we have $p(x_{m_n}-x)\downarrow 0$ by applying \cite[Thm.3]{SP}. Hence, $p(x_{m_n}-x)\std 0$, and so, $p(x_{m_n}-x)\sto 0$, i.e., $(x_{m_n}-x)\stpc 0$. Therefore, we get $x_{m_n}\stupc x$.
\end{proof}

Recall that a Dedekind complete Riesz space $X^\delta$ is said to be a Dedekind completion of a Riesz space $X$ whenever $X$ is Riesz isomorphic to a majorizing order dense Riesz subspace of $X^\delta$. It is well known that every Archimedean Riesz space has a Dedekind completion (cf. \cite[Thm.2.24]{ABPO}). A sequence $x_n\oc 0$ in a Riesz space $X$ if and only if $x_n\oc 0$ in $X^\delta$ (cf. \cite[Cor.2.9]{GTX}).
\begin{theorem}\label{$up$-conv by $p$-unit}
	Let $(X,p,E)$ be an $LNRS$ and $e$ be a $p$-unit in $X$. Then we have $x_n\stupc0$ in $X$ if and only if $p(|x_n|\wedge e)\sto 0$ in $E$.
\end{theorem}

\begin{proof}
	Suppose that $x_n\stupc0$ holds in $X$. Then, by taking $u=e\in E_+$, we get $p(|x_n|\wedge e)\sto 0$ in $E$. 
	
	Now, assume that $p(|x_n|\wedge e)\sto 0$ in $E$. Then there exists a sequence $q_n\stpd0$ with a set $\delta(K)=1$ such that $p(|x_{k_n}|\wedge e)\leq q_{k_n}$ for all $k_n\in K$, i.e., we have $p(|x_{k_n}|\wedge e)\oc 0$ on $K$. Take any $u \in X_+$. Then we observe that
	\begin{eqnarray*}
		p(|x_{k_n}|\wedge u)&\leq& p(|x_{k_n}|\wedge(u-u\wedge ne))+p(|x_{k_n}|\wedge(u\wedge ne))\\&\leq& p(u-u\wedge ne)+np(|x_{k_n}|\wedge e)	
	\end{eqnarray*} 
	holds in $E^\delta$ for each $k_n\in K$ and for every $n\in\mathbb{N}$. So, we have
	$$
	\limsup\limits_{k_n} p(|x_{k_n}|\wedge u)\leq p(u-u\wedge ne)+n\limsup\limits_{k_n} p(|x_{k_n}|\wedge e)
	$$ 
	holds in $E^\delta$ for all $n\in\mathbb{N}$. Since $p(|x_{k_n}|\wedge e)\oc0$ in $E$, we have $p(|x_{k_n}|\wedge e)\oc0$ in $E^\delta$. Thus, we see $\limsup\limits_{k_n}p(|x_{k_n}|\wedge e)=0$ in $E^\delta$. Thus
	$$
	\limsup\limits_{k_n}  p(|x_{k_n}|\wedge u)\leq p(u-u\wedge ne)
	$$
	holds in $E^\delta$ for all $n\in\mathbb{N}$. Since $e$ is a $p$-unit, we have $\limsup\limits_{k_n} p(|x_{k_n}|\wedge u)=0$ in $E^\delta$ or $p(|x_{k_n}|\wedge u)\oc 0$ in $E^\delta$. It follows that $p(|x_{k_n}|\wedge u)\oc 0$ in $E$. Hence, $x_n\stupc0$. 
\end{proof}

\begin{theorem}
	Let $(X,p,E)$ be an $LNRS$. Then define a Riesz norm $p^{\delta}:X^{\delta}\rightarrow E^{\delta}$ by $p^{\delta}(z)=\sup\limits_{0\leq x\leq |z|}p(x)$ for every $z\in X^\delta$. Then we have $x_n\stupc x$ in $(X,p,E)$ if and only if $x_n\stupc x$ in $(X^\delta,p^\delta,E^\delta)$.
\end{theorem}

\begin{proof}
	It follows from \cite[Prop.4.1]{AEEM1} that $(X^\delta,p^\delta,E^\delta)$ is an $LNRS$.
	
	Take a sequence $x_n\stupc0$ in $(X,p,E)$ and fix an arbitrary positive element $z$ in $X^\delta_+$. Then, for any positive element $u\in X_+$, there exists a sequence $q_n\stpd0$ in $E$ with a set $\delta(K)=1$ such that $p(|x_{k_n}-x|\wedge u)\leq q_{k_n}$ for every $k_n\in K$. On the other hand, since $X$ is majorizing in its Dedekind completion $X^\delta$, there exists $u\in X$ such that $0<z\leq u$. By using $|x_{k_n}-x|\wedge z\leq |x_{k_n}-x|\wedge u$, we obtain 
	$$
	p^\delta(|x_{k_n}-x|\wedge z)\leq p(|x_{k_n}-x|\wedge u)\leq q_{k_n}
	$$
	for each $k_n\in K$. Thus, $x_n\stupc x$ in $(X^\delta,p^\delta,E^\delta)$ because $q_{k_n}\downarrow0$ in $X$ implies $q_{k_n}\downarrow0$ in $X^\delta$; \cite[Cor.2.9]{GTX}.
	
	For the converse, suppose that $x_n\stupc 0$ in $(X^\delta,p^\delta,E^\delta)$ and $u$ be a positive element in $X_+$. Then there exists a sequence $q_n\stpd0$ in $E^\delta$ with a set $\delta(K)=1$ such that $p^\delta(|x_{k_n}-x|\wedge u)\leq q_{k_n}$ for every $k_n\in K$. Again by applying \cite[Cor.2.9]{GTX}, $q_{k_n}\downarrow0$ in $X^\delta$ implies $q_{k_n}\downarrow0$ in $X$. So, we have 
	$$
	p(|x_{k_n}-x|\wedge u)=p^\delta(|x_{k_n}-x|\wedge u)\leq q_{k_n}
	$$
	for every $k_n\in K$ because of $(x_n)$ in $X$. It follows that $x_n\stupc x$ in $(X,p,E)$.
\end{proof}

\begin{proposition}
	A statistical order convergent sequence in $op$-continuous $LNRS$s is $st$-$up$-convergent.	
\end{proposition}

\begin{proof} 
	Let $x_n\sto x$ be in an $op$-continuous $LNRS$ $(X,p,E)$. Then there exists a sequence $s_n\std0$ in $X$ with an index set $I$ such that $\delta(I)=1$ and $|x_{i_n}-x|\leq s_{i_n}$ for each $i_n\in I$. Hence, we get $p(x_{i_n}-x)\leq p(s_{i_n})$, and so, $p(|x_{i_n}-x|\wedge u)\leq p(s_{i_n})$ for every $i_n\in I$ and each $u\in X_+$. By using $s_{i_n}\downarrow0$ on $I$ in $X$, it follows from $op$-continuity of $(X,p,E)$ that $p(s_{i_n})\downarrow0$ on $I$ in $E$. Therefore, we get the desired result, $x_n\stupc 0$.
\end{proof}

Recall that a subset $A$ of a Riesz space $E$ is called solid if, for each $x\in A$ and $y\in E$,  $|y|\leq|x|$ implies $y\in A$. Also, a solid vector subspace of a Riesz space is referred to as an ideal. Moreover, an order closed ideal is called a band (cf. \cite{AB,AlTo}).
\begin{proposition}
	Let $B$ be a band in an $LNRS$ $(X,p,E)$. If $b_n\stupc x$ is a sequence in $B$ then $x\in B$.
\end{proposition}

\begin{proof}
	Assume that $b_n\stupc x$ satisfies for a sequence $(b_n)$ in $B$ and some $x\in X$. By applying Theorem \ref{LO are continuous}, we have $|b_n|\wedge|z| \stupc |x|\wedge|z|$ for any $z\in B^d:=\{z\in X: |z|\wedge|b|=0\ \text{for \ all} \ b\in B\}$. Also, $|b_n|\wedge|z|=0$ for all $n$ because $(b_n)$ is a sequence in $B$. Therefore, we obtain $|x|\wedge|z|=0$. So, $x\in B^{dd}$. As a result, it follows from \cite[Thm.1.39]{ABPO} that we have $B=B^{dd}$, and so, we obtain $x\in B$.
\end{proof}

\begin{proposition}
	Let $B$ be a projection band in an $LNRS$ $(X,p,E)$ and $p_B$ be the corresponding band projection of $B$. Then $x_n\stupc x$ implies $P_B(x_n)\stupc P_B(x)$. 
\end{proposition}

\begin{proof}
	Let $x_n\stupc x$ be in an $LNRS$ $(X,p,E)$. Then, for any $u\in X_+$, there exists a sequence $q_n\stpd 0$ with a set $\delta(K)=1$ such that $p(|x_{k_n}-x|\wedge u)\leq q_{k_n}$ for all $k_n\in K$. Also, it is well known that a band projection $p_B$ is a lattice homomorphism and it satisfies the inequality $0\leq P_B\leq I$. Hence, we have 
	$$	p(|P_B(x_{k_n})-P_B(x)|\wedge u)=p((P_B|x_{k_n}-x|)\wedge u)\leq p(|x_{k_n}-x|\wedge u)\leq q_{k_n}	$$
	for every $k_n\in K$. Thus, we have $P_B(x_n)\stupc P_B(x)$.
\end{proof}

\begin{theorem}\label{$up$-regular}
	Let $(X,p,E)$ be an $LNRS$ and $Y$ be a sublattice of $X$. If a sequence $(y_n)$ in $Y$ is $st$-$up$-convergent to zero in $Y$ then it is $st$-$up$-convergent to zero in $X$ for each of the following cases:
	\begin{enumerate}
		\item[(i)] $Y$ is majorizing in $X$;
		\item[(ii)] $Y$ is $p$-dense in $X$;
		\item[(iii)] $Y$ is a projection band in $X$.
	\end{enumerate}
\end{theorem}

\begin{proof}
	Assume that $(y_n)$	is a sequence in $Y$ such that $y_n\stupc 0$ in $Y$ and $u$ be a positive element in $X_+$.
	
	$(i)$ Since $Y$ is majorizing in $X$, there exists $y\in Y$ such that $u\leq y$. It follows from
	$$
	0\leq p(|y_n|\wedge u)\leq p(|y_n|\wedge y)\sto 0,
	$$
	that $p(|y_n|\wedge u)\sto 0$, i.e., $y_n\stupc 0$ in $X$.
	
	$(ii)$ Take any $0\ne w\in p(X)$. Thus, there is $y\in Y$ with $p(u-y)\le w$. So, it follows that
	$$
	p(|y_n|\wedge u)\le p(|y_n|\wedge|u-y|)+p(|y_n|\wedge |y|)\leq w+p(|y_n|\wedge |y|).
	$$
	Then $p(|y_n|\wedge u)\sto 0$ because $0\ne w\in p(X)$ is arbitrary and $p(|y_n|\wedge |y|)\sto0$. So, we get $y_n\stupc0$ in $X$.
	
	$(iii)$ It is clear that $Y=Y^{\bot\bot}$ implies $X=Y\oplus Y^{\bot}$. Thus, $u=u_1+u_2$ with $u_1\in Y$ and $u_2\in Y^{\bot}$. It follows from  $y_n\wedge u_2=0$ and \cite[Lem.1.4]{ABPO} that we have 
	$$
	p(|y_n|\wedge u)=p(|y_n|\wedge(u_1+u_2))\leq p(|y_n|\wedge u_1)\sto0.
	$$ 
	As a result, we get $y_n\stupc0$ in $X$.
\end{proof}



\begin{thebibliography}{99}
	\bibitem{AS}
	Abramovich YA, Sirotkin G. On order convergence of nets. Positivity, 2005; 9 (3): 287-292. doi.org/10.1007/s11117-004-7543-x
	
	\bibitem{AB}
	Aliprantis CD, Burkinshaw O.
	Locally Solid Riesz Spaces with Applications to Economics. Mathematical Surveys and Monographs Centrum, 2003.
	
	\bibitem{ABPO}
	Aliprantis CD, Burkinshaw O. Positive Operators. Dordrecht, Springer, 2006.
	
	\bibitem{AlTo}
	Aliprantis CD, Tourky R. Cones and Duality. Graduate Studies in Mathematics, vol. 84, American Mathematical Society, Providence, RI, 2007.
	
	\bibitem{Aydn1}
	Ayd{\i}n A.
	The statistically unbounded $\tau$-convergence on locally solid Riesz spaces. Turkish Journal of Mathematics 2020; 44 (3): 949-956. doi.org/10.3906/mat-1912-37
	
	\bibitem{Aydn2}
	Ayd{\i}n A. The statistical multiplicative order convergence in vector lattice algebras. Facta Universitatis, Series: Mathematics and Informatics 2021; 36 (2): 409-417. doi.org/10.22190/FUMI200916030A
	
	\bibitem{Aydn3}
	Ayd\i n A. Statistical unbounded order convergence in Riesz spaces, submitted.
	
	\bibitem{AYE}
	Ayd\i n A, Yapal\i\ R, Korkmaz E. Statistical $p$-convergence in lattice-normed Riesz spaces, arxiv.org/abs/2105.08420v1.
	
	\bibitem{AEEM1}
	Ayd\i n A, Emelyanov EY, \"Ozcan NE, Marabeh MAA.
	Unbounded $p$-convergence in lattice-normed vector lattices.
	Siberian Advances in Mathematics 2019; 29: 164-182. doi.org/10.3103/S1055134419030027
	
	\bibitem{DOT}
	Deng Y, O'Brien M. Troitsky VG. Unbounded Norm Convergence in Banach Lattices. Positivity 2017; 21: 963-974. doi.org/10.1007/s11117-016-0446-9
	
	\bibitem{E}   	
	Emelyanov EY. Infinitesimal analysis and vector lattices. Siberian Advances in Mathematics 1996; 6 (1): 19-70.
	
	\bibitem{Ercan}
	Ercan Z. A characterization of $u$-uniformly completeness of Riesz spaces in terms of statistical $u$-uniformly pre-completeness. Demonstratio Mathematica 2009; 42 (2): 381–385. doi.org/10.1515/dema-2009-0215
	
	\bibitem{Fast}
	Fast H. Sur la convergence statistique. Colloquium Mathematicae 1951; 2: 241-244.
	
	\bibitem {Fridy}
	Fridy JA.
	On statistical convergence.
	Analysis 1985; 5 (4): 301-313. doi.org/10.1524/anly.1985.5.4.301
	
	\bibitem{GTX}
	Gao N, Troitsky VG,  Xanthos F. $Uo$-convergence and its applications to Ces\'aro means in Banach lattices. Israel Journal of Mathematics 2017; 220: 649–689. doi.org/10.1007/s11856-017-1530-y
	
	\bibitem{Gor}
	Gorokhova SG. Intrinsic characterization of the space $c_0(A)$ in the class of Banach lattices. Mathematical Notes 1996; 60: 330-333.  doi.org/10.4213/mzm1846
	
	\bibitem{KK}
	Kusraev AG, Kutateladze SS. Boolean Valued Analysis. Mathematics and its Applications, 1999.
	
	\bibitem{K}
	Kusraev AG. Dominated Operators. Mathematics and its Applications,
	2000.
	
	\bibitem{LZ}
	Luxemburg WAJ, Zaanen AC. Riesz Spaces I. North-Holland Pub. Co., Amsterdam, 1971.
	
	\bibitem{N}
	Nakano H.
	Ergodic theorems in semi-ordered linear spaces.
	Annals of Mathematics 1948; 49 (3): 538-556. doi:10.2307/1969044
	
	\bibitem{Riez}
	Riesz F.
	Sur la Décomposition des Opérations Fonctionelles Linéaires. Bologna, Atti Del Congresso Internazionale Dei Mathematics Press, 1928. 
	
	\bibitem{St}
	Steinhaus H.
	Sur la convergence ordinaire et la convergence asymptotique.
	Colloquium Mathematicum 1951; 2: 73-74.
	
	\bibitem{SP}
	Şençimen C, Pehlivan S. Statistical order convergence in Riesz spaces. Mathematica Slovaca 2012; 62 (2): 557-570. doi:10.2478/s12175-012-0007-z
	
	\bibitem{T}
	Troitsky VG. Measures of non-compactness of operators on Banach lattices. Positivity 2004; 8 (2): 165-178. 
	
	\bibitem{Za}
	Zaanen AC.
	Riesz Spaces II. Amsterdam, The Netherlands: North-Holland Publishing, 1983.
	
	\bibitem{Zygmund}
	Zygmund A. Trigonometric Series, Cambridge University Press, Cambridge, UK, 1979.
\end{thebibliography}
\end{document}